\newcommand{\fsize}{10pt} 
\documentclass[\fsize]{amsart}
\usepackage{amsmath,amssymb}
\usepackage{enumerate, amscd, color}
\usepackage{slashed}
\usepackage{MnSymbol}

\newcommand\arxiv[1]{\href{https://www.arxiv.org/abs/#1}{arXiv:~#1}}

\newcommand\doi[1]{\href{https://doi.org/#1}{DOI:~#1}}

\newif\ifdraft
\draftfalse 

\newif\iflong
\longfalse

\newif\ifarxivversion
\arxivversiontrue

\ifdraft
\newcounter{mnotecount}[section]
\renewcommand{\themnotecount}{\thesection.\arabic{mnotecount}}
\newcommand{\mnote}[1]{\protect{\stepcounter{mnotecount}}${\raisebox{0.5\baselineskip}[0pt]{\makebox[0pt][c]{\tiny\em{\red{$\bullet$\themnotecount}}}}}$\marginpar{\raggedright\tiny\em $\!\!\!\!\!\!\,\bullet$\themnotecount: #1}\ignorespaces}
\else 
\newcommand{\mnote}[1]{}
\fi

\usepackage{xcolor}

\definecolor{darkgreen}{rgb}{0,0.6,0}
\definecolor{darkred}{rgb}{0.7,0,0}
\definecolor{darkblue}{rgb}{0,.1,.6}
\definecolor{darkgray}{rgb}{0.3,.3,.3}
 \usepackage[%
   colorlinks, 
  citecolor=darkgreen, linkcolor=darkblue,
   pdfpagelabels=true,
   unicode=true,
   pdfusetitle
 ]{hyperref}

\newcommand\red[1]{{\color{red}{#1}}}

\ifdraft
\usepackage{showkeys, booktabs}
\fi

\usepackage[normalem]{ulem}
\renewcommand\sout{\bgroup\markoverwith
{\textcolor{red}{\rule[0.7ex]{3pt}{1.4pt}}}\ULon}
 
\let\epsilon\varepsilon

\let\theta\vartheta

\def\phi{{\varphi}}

\let\<\langle
\let\>\rangle
\newcommand\ie{i.\thinspace e.\ \ignorespaces}
\newcommand\eg{e.\thinspace g.\, \ignorespaces}
\newcommand\Dirac{\slashed{D}}

\newcommand\cE{\mathcal{E}}
\newcommand\cR{\mathcal{R}}
\newcommand\tr{\mathrm{tr}}

\newcommand\Ker{\mathrm{ker}}
\newcommand\Cl{\mathrm{Cl}}
\newcommand\dvol{\,\mathrm{dvol}}
\newtheorem{theorem}{Theorem}
\newtheorem{corollary}[theorem]{Corollary}

\theoremstyle{remark}
\newtheorem{remark}[theorem]{Remark}
\newtheorem*{remark*}{Remark}

\newlength{\keyheightlength}
\newcommand{\stellle}[1]{\mathchoice%
  {\mbox{\lower.15\keyheightlength\hbox{\upshape|}}_{#1}}%
  {\mbox{\lower.12\keyheightlength\hbox{\upshape|}}_{#1}}%
  {\mbox{\lower.1\keyheightlength\hbox{\upshape|}}_{#1}}%
  {\mbox{\lower.09\keyheightlength\hbox{\upshape|}}_{#1}}%
}
\newcommand{\stelle}[1]{\mathchoice%
  {\lower.27\keyheightlength\hbox{\big|}\lower.57\keyheightlength\hbox{$\scriptstyle #1$}}%
  {\lower.24\keyheightlength\hbox{\big|}\lower.48\keyheightlength\hbox{$\scriptstyle #1$}}%
  {\lower.21\keyheightlength\hbox{|}\lower.42\keyheightlength\hbox{$\scriptscriptstyle #1$}}%
  {\lower.18\keyheightlength\hbox{|}\lower.39\keyheightlength\hbox{$\scriptscriptstyle #1$}}%
}
\newcommand{\Stelle}[1]{\mathchoice%
  {\lower.39\keyheightlength\hbox{\Big|}\lower.74\keyheightlength\hbox{$\scriptstyle #1$}}%
  {\lower.33\keyheightlength\hbox{\big|}\lower.66\keyheightlength\hbox{$\scriptstyle #1$}}%
  {\lower.30\keyheightlength\hbox{\big|}\lower.60\keyheightlength\hbox{$\scriptscriptstyle #1$}}%
  {\lower.25\keyheightlength\hbox{\big|}\lower.51\keyheightlength\hbox{$\scriptscriptstyle #1$}}%
}
\newcommand{\sstelle}[1]{\mathchoice%
  {\lower.39\keyheightlength\hbox{\bigg|}\lower1.2\keyheightlength\hbox{$\scriptstyle #1$}}%
  {\lower.32\keyheightlength\hbox{\Big|}\lower.74\keyheightlength\hbox{$\scriptstyle #1$}}%
  {\lower.32\keyheightlength\hbox{\big|}\lower.75\keyheightlength\hbox{$\scriptstyle #1$}}%
  {\lower.29\keyheightlength\hbox{\big|}\lower.63\keyheightlength\hbox{$\scriptstyle #1$}}%
}
\newcommand{\SStelle}[1]{\mathchoice%
  {\lower.40\keyheightlength\hbox{\Bigg|}\lower1.35\keyheightlength\hbox{$\scriptstyle #1$}}%
  {\lower.37\keyheightlength\hbox{\bigg|}\lower.98\keyheightlength\hbox{$\scriptstyle #1$}}%
  {\lower.35\keyheightlength\hbox{\Big|}\lower.90\keyheightlength\hbox{$\scriptscriptstyle #1$}}%
  {\lower.30\keyheightlength\hbox{\big|}\lower.72\keyheightlength\hbox{$\scriptscriptstyle #1$}}%
}

\begin{document}

\title{Are all Dirac-harmonic maps uncoupled?}
\author[B. Ammann]{Bernd Ammann} \address{B. Ammann, Fakult\"at f\"ur
  Mathematik, Universit\"at Regensburg, 93040 Regensburg, Germany}
\email{bernd.ammann@mathematik.uni-regensburg.de}

\thanks{The author was supported by SPP 2026 (Geometry at
  infinity) and the SFB 1085 (Higher Invariants), both funded by the
  DFG (German Science Foundation).}

\begin{abstract}Dirac-harmonic maps $(f,\phi)$ consist of a map $f:M\to N$ and a twisted spinor $\phi\in\Gamma(\Sigma M\otimes f^*TN)$ and they are defined as critical points of the super-symmetric energy functional.
A Dirac-harmonic map is called \emph{uncoupled}, if $f$ is a harmonic map.
We show that under some minimality assumption Dirac-harmonic maps defined
on a closed domain are uncoupled.
\end{abstract}

\date\today

\maketitle

\section{Introduction}

Let $M$ and $N$ be Riemannian manifolds, and $f:M\to N$ a $C^1$. If $M$ is compact, we can define the energy of $f$ as
  $$\cE_1(f) = \frac12 \int_M |df|^2\dvol^g$$
considered as a functional on $C^1(M,N)$. Critical points of this functional are called \emph{harmonic maps}, and there is an extensive literature about harmonic maps, with many interesting applications, also weak solutions were studied.

In the recent years, there is a growing number of publications about a super-symmetric analogue of harmonic maps, called \emph{Dirac-harmonic maps}. They are defined as critical points of the sypersymmetric energy functional $\mathcal{E}$ defined in \eqref{eq.energy.supersym}, see Section~\ref{sec.main.idea} for more details.
In particular, a Dirac-harmonic map is a pair $(f,\phi)$ of a map $f:M\to N$ and a twisted spinor $\phi\in\Gamma(\Sigma M\otimes f^*TN)$. An early article~\cite{chen.jost.li.wang:05} about such maps was written in 2005 by Chen, Jost, Li and Wang, studying regularity issues for Dirac-harmonic maps, followed by \cite{chen.jost.li.wang:06,chen.jost.wang:07} and stimulating many associated questions and results. Between 2005 and September 2022, MathSciNet lists about 45 publications with ``Dirac-harmonic'' in the title. Many questions answered for harmonic maps can be discussed in the Dirac-harmonic context, often it is hard to include the spinorial part into the estimates, and good progress was achieved.

Obiously, Dirac-harmonic maps with $\phi\equiv 0$ are uninteresting, as then $(f,\phi)$ is Dirac-harmonic if and only if $f$ is harmonic; such solution are called \emph{spinor-trivial}. Similarly, solutions with $f$ constant, called \emph{map-trivial} solutions, are not in the focus of interest within this subject either; in this case the problem is equivalent to finding harmonic spinors in the classical sense, see \eg \cite{hitchin:74,baer:96b,ammann.dahl.humbert:09}. We say that a Dirac-harmonic map is \emph{uncoupled} if $f$ is harmonic, otherwise we say that this Dirac-harmonic map is \emph{coupled}.
Being uncoupled is equivalent to the vanishing of the $\cR$-term, see Section~\ref{sec.main.idea}. A major question discussed in this article is whether coupled Dirac-harmonic maps with compact domain exist. Let us discuss examples of Dirac-harmonic maps in the literature first.

Progress in constructing Dirac-harmonic maps was achieved in \cite{jost.mo.zhu:09} by using (untwisted) harmonic spinors, twistor spinors and similar solutions of other spinorial equations, although it remained unclear, how one could get solutions of these spinorial equations. This was analyzed later by Ginoux and the author in \cite{ammann.ginoux:19}: we showed -- see \cite[Theorems 1.1 and~1.3]{ammann.ginoux:19} -- that the conditions in \cite[Theorems~1 and 3]{jost.mo.zhu:09} can be satisfied only in exceptional cases.\footnote{Note that the proof of \cite[Theorems~1]{jost.mo.zhu:09} uses that the immersion is isometric, although this is not stated in this theorem \cite[Theorems~1]{jost.mo.zhu:09} explicitly.} In particular, for $\dim M\geq 3$ and $M$ complete, one concludes that~$M$ has to be simply-connected of constant negative curvature. Other strong obstructions exist for  $\dim M=2$, in particular $f$ has to be necessarily harmonic, thus any Dirac-harmonic map with compact domain obtained this way is uncoupled. The solutions in \cite[Theorems~2]{jost.mo.zhu:09}, reproven as \cite[Corollary~2.3]{ammann.ginoux:19} are uncoupled as well.

Many more Dirac-harmonic maps $(f,\phi)$ were constructed in  \cite{ammann.ginoux:13}. Here, the domain $M$ of $f$ is closed. The method starts with a harmonic map $f:M\to N$ and then one uses index theory to get a non-vanishing harmonic spinor $\phi\in \Gamma(\Sigma M\otimes f^*T^*N)$. In particular, these solutions are uncoupled.

In the special case $M=N=S^2$ with the round metric, it was shown in \cite{yang_ling:09} that every Dirac-harmonic map is uncoupled, and this was recently reproven in  \cite[Proposition 1.1]{jost.sun.zhu:p22}.

Summarizing this, the author is not aware of any publication in which the existence of any coupled Dirac-harmonic map with compact domain was proven.

In the current article, we will show, that ``generically'' every Dirac-harmonic map with closed domain is uncoupled, see Corollary~\ref{dirac-harm.trivial}  for details. The argument may be adapted to compact manifolds $M$ with boundary for many suitable homogeneous boundary conditions, an extension of the results that we will not work out.\mnote{more details} 

It thus remains questionable whether coupled Dirac-harmonic maps with closed domain actually exist. Suppose such solutions did not exist, then the main results in several recent publications would not provide new result, compared to what is known already for harmonic maps. However, spectral flow methods might possibly be used to construct coupled Dirac-harmonic maps with closed domain. Thus it is hard to predict whether an overseen reason implies $\cR=0$ for $M$ closed, or whether we did not yet find the right tools to find coupled solutions.

\textbf{What we do not discuss here.}
In order to avoid missunderstandings, we want to add here some issues that we \emph{do not} prove in this article. First, there is a variant of the super-symmetric energy functional that includes a quartic spinor term, depending on the curvature of the target manifold $N$, see \eg \cite[(2.4)]{jost.liu.zhu:22}. Orally, I was told, that this variant is even more important from the perspective of applications to physics. I am unable to adapt the arguments of the current article to this modified functional. On the other hand I expect that
many publications about Dirac-harmonic maps are sufficiently robust in order to be applicable also to critical points of this modified functional.
Similar arguments apply to other perturbations of the supersymmetric energy functional. As a consequence, the techniques in the articles listed above would provide valuable contributions even when a non-existence result for coupled solution on closed domains were finally proven.

Second, we did not yet discuss the case of non-compact domains. Our argument relies essentially on the fact that the Euler-Lagrange equations \eqref{eq.lagrange} describe the stationary points of a well-defined functional, that $\Dirac^f$ is self-adjoint and that we do not get boundary terms by partial integration. In \cite{jost.mo.zhu:09} and \cite{ammann.ginoux:19} an example of a Dirac-harmonic map with non-vanishing $\cR$-term is (implicitly) given: here $M^m$ is a simply connected complete Riemannian manifold of constant negative sectional curvature $-4/(m+2)$ (\ie a space form), and $f$ is an isometric embedding into $(m+1)$-dimensional hyperbolic space $N$, for which $f(M)$ is totally umbillic with parallel shape tensor in $N$. Due to the conformal covariance of the Dirac operator and of the Penrose operator, $M$ carries many twistor spinors and harmonic spinors, thus the Jost-Mo-Zhu method \cite{jost.mo.zhu:09}  may be applied, and one easily checks that $\tau(f)$ is a non-zero constant.
However, in this case $|df|^2$ is not integrable, and thus the super-symmetric energy functional $\cE$ does not converge. The Dirac-harmonic map $(f,\phi)$ solves the system of partial differential equations \eqref{eq.lagrange} without being the stationarity equation of a functional defined on all pairs of maps with spinors.

Third, we only considered solutions to the Dirac-harmonic map equation in the strong sense. We do not know to which extend our methods also generalize to weak settings. For the harmonic maps, weak solutions gave rise to involved research, see e.g.\ Bethuel's work \cite{bethuel-manus:93}.

\textbf{Acknowledgements.}
The result of this note grew out of discussions related to a pair of conferences in Shanghai and Sanya in spring 2018, organized by J\"urgen Jost, Miaomiao Zhu and others. My thanks goes to these organizers. I also thank my former PhD student Johannes Wittmann for his contributions. Thanks also to Christian B\"ar and Johannes Wittmann for comments on an earlier version of this preprint. Thanks also to Lei Liu for pointing out that I had overseen a boundary term in a previous version.
\section{Dirac-harmonic maps with closed domain are generically uncoupled}\label{sec.main.idea}

Let $(M,g)$ be a compact Riemannian spin manifold, and $(N,h)$ a Riemannian manifold. For a map $f:M\to N$ and a twisted spinor $\phi\in \Gamma(\Sigma M\otimes f^*TN)$ we define
  \begin{eqnarray}
   \cE_1(f) &:=& \frac12 \int_M |df|^2\dvol^g\nonumber\\
   \cE_2(f,\phi) &:=&  \frac12 \int_M \< \phi,\Dirac^f\phi\> \dvol^g\nonumber\\
   \cE(f,\phi) &:=&  \cE_1(f)  + \cE_2(f,\phi) \,.\label{eq.energy.supersym}
  \end{eqnarray}
  The functional $\mathcal{E}$ is called the \emph{super-symmetric energy functional}.
  We define
  \begin{eqnarray*}
   \tau(f) &:=& -\frac{\partial \cE_1}{\partial f}=  \tr \nabla d f \in \Gamma(f^*TN)\\
    \cR(f,\phi) &:=& \frac{\partial \cE_2}{\partial f}\\
    \Psi(f,\phi)  &:=& \frac{\partial \cE_2}{\partial \phi}=\Dirac^f\phi\,.
  \end{eqnarray*}
  In these equations, $\Dirac^f$ denotes the twisted Dirac operator acting on  $\Gamma(\Sigma M\otimes f^*TN)$. Further, we used the canonical $L^2$-scalar product on $\Gamma(f^*TN)$ and on  $\Gamma(\Sigma\otimes f^*TN)$, in order to identify the partial derivatives of $\cE_1$ and $\cE_2$, given above, with the corresponding gradients.
  
  Stationary points of $\cE$ are called \emph{Dirac-harmonic maps}. Thus this condition is equivalent to
  \begin{equation}\label{eq.lagrange}
    \Dirac^f\phi=0\text{ and }\tau(f)= \cR(f,\phi).
  \end{equation}
  We say that $(f,\phi)$ is \emph{uncoupled}, if  $\cR(f,\phi)=0$. In this publication we only consider \emph{smooth} {Dirac-harmonic maps.  
  \begin{remark*}
    Note that $\Sigma M\otimes f^*TN $ will always denote the real tensor product, even when  $f^*TN$ and $\Sigma M$ carry natural complex structures.
    Note that  $\Sigma M$ may be defined as a real, a complex, a quaternionic or a $\Cl_m$-linear Dirac operator, this does not matter. For simplicity we only restrict to the complex version in this article, as it is classically the most studied one.
  \end{remark*}
  \begin{remark*}
  The terminology for the $\cR$-term is standard, see e.g. \cite[Page 1514, after (9)]{jost.mo.zhu:09}, \cite[(1.3)]{chen.jost.sun.zhu:19}.
\end{remark*}

\begin{theorem}\label{R.null}
  Let $(f_0,\phi_0)$ be given. We assume that there is an open neighborhood $U$ of $f_0$, and a $C^1$-map $U\ni f\mapsto \hat\phi(f)$, such that $\Dirac^f(\hat\phi(f))=0$ and such that  $\hat\phi(f_0)=\phi_0$.
Then $(f_0,\phi_0)$ is uncoupled. 
\end{theorem}

\begin{proof}
  Obviously we have $\cE_2(f,\hat\phi(f))=0$. In the following we write $d/df$ for the total\footnote{In the sense used in Lagrangian mechanics: we derive the expression $f\mapsto  \cE(f,\hat\phi(f))$} derivative with respect to $f$ and -- as before -- $\partial/\partial f$ for the partial derivative with respect to~$f$. We thus get
  \begin{eqnarray*}
    -\tau(f) &=& \frac{d}{df} \cE_1(f)\\
     &=& \frac{d}{df} \Bigl(\cE\bigl(f,\hat\phi(f)\bigr)\Bigr)\\
     &=& \frac{\partial\cE}{\partial f}\Stelle{f}  + \frac{\partial\cE }{\partial \phi}\Stelle{(f,\hat\phi(f))} \, \frac{\partial\hat\phi }{\partial f}\Stelle{f} \\
     &=& \frac{\partial\cE_1}{\partial f}\Stelle{f} +\frac{\partial\cE_2}{\partial f}\Stelle{(f,\hat\phi(f))}  + \underbrace{\frac{\partial\cE_2}{\partial \phi}\Stelle{(f,\hat\phi(f))} }_{=\Psi(f,\hat\phi(f))=0}\frac{\partial\hat\phi }{\partial f}\Stelle{f}\\
    &=& -\tau(f) +\cR(f,\hat\phi(f)).
  \end{eqnarray*}
  and this obviously implies $\cR(f_0,\phi_0)=0$.
\end{proof}

\begin{corollary}
  Suppose $f_0$ has a neighborhood $U$ such that for all $f\in U$ we have
  $\dim \Dirac^f\geq \dim \Dirac^{f_0}$. Then for every $\phi_0\in \Ker \Dirac^{f_0}$  the pair  $(f_0,\phi_0)$ is uncoupled.
\end{corollary}

Or as a special case we get.
\begin{corollary}\label{dirac-harm.trivial} 
  Suppose $(f_0,\phi_0)$  is a Dirac-harmonic map with $M$ closed. Then
  \begin{enumerate}
    \item $(f_0,\phi_0)$ is uncoupled, or 
    \item any neighborhood $U$ of $f_0$ contains an $f\in U$ with $\dim \Dirac^f< \dim \Dirac^{f_0}$.
    \end{enumerate}
\end{corollary}

This corollary shows that the construction of coupled Dirac-harmonic maps with compact domain is difficult. This fits to the fact that to the author's knowledge no coupled Dirac-harmonic map with closed domain has been found so far.

Note that these results may be easily generalized to compact manifolds with boundary with suitable homogeneous boundary conditions. \mnote{More here}However, our methods are not robust enough to generalize to more general type of functionals $\cE$, \eg to the one with quartic spinor term discussed in \cite[(2.4)]{jost.liu.zhu:22}.

\begin{remark}For studying harmonic maps, Sacks and Uhlenbeck \cite{sacks.uhlenbeck:81} applied a perturbation of the energy functional successfully, namely they defined for $\alpha\in (1,2)$
  $$\cE_\alpha(f) = \frac12 \int_M (1+|df|^2)^\alpha\dvol^g.$$
Recently, a similar perturbation was also used for the super-symmetric energy functional, \ie one defines $\cE(f,\phi)=\cE_\alpha(f)+\cE_2(f,\phi)$, see \eg \cite[Eq. (1.2)]{jost.zhu:21CalcVar}, \cite{jost.zhu:21} and \cite{li.liu.zhu.zhu:21}.  
Obviously Theorem~\ref{R.null} remains true for this modification. The condition of Theorem~\ref{R.null} are satisfied as index theoretical methods and minimal kernel methods are used. As a consequence all $\alpha$-Dirac-harmonic maps in \cite{jost.zhu:21} are uncoupled and most of the results for ($\alpha$)-Dirac-harmonic maps in \cite{jost.zhu:21} directly follow from the corresponding statements for ($\alpha$)-harmonic maps of the same article. In the results in  \cite[Eq. (1.1)]{jost.zhu:21CalcVar} an additional perturbation $F$ is allowed, which makes our trick non-applicable, while it applies to stationary points of (1.2) in the same article. Again, it seems unanswered whether coupled $\alpha$-Dirac-harmonic maps with closed domain exist.
\end{remark}


\end{document}